\documentclass[11pt]{article}
\usepackage{epigamath}

\usepackage[notext]{kpfonts}
\usepackage{baskervald}

\setpapertype{A4}

\usepackage[english]{babel}

\usepackage[dvips]{graphicx}     

\title{\vspace{-0.5cm}On the Prym variety of genus $3$ covers of genus $1$ curves}
\titlemark{On the Prym variety}
\author{\vspace{0cm}Christophe Ritzenthaler and Matthieu Romagny}
\authoraddresses{
\authordata{Christophe Ritzenthaler}{\firstname{Christophe} \lastname{Ritzenthaler}\\
\institution{Universit\'e de Rennes 1, CNRS, IRMAR - UMR 6625, F-35000, France}\\
\email{christophe.ritzenthaler@univ-rennes1.fr}}\\
\authordata{Matthieu Romagny}{\firstname{Matthieu}
\lastname{Romagny}\\
\institution{Universit\'e de Rennes 1, CNRS, IRMAR - UMR 6625, F-35000, France}\\
\email{matthieu.romagny@univ-rennes1.fr}}
}
\authormark{C. Ritzenthaler and M. Romagny}
\date{\vspace{-5ex}} 
\journal{\'Epijournal de G\'eom\'etrie Alg\'ebrique} 
\acceptation{Received by the Editors on May 18, 2017, and in final form
on December 26, 2017. \\ Accepted on March 9, 2018.}

\newcommand{\articlereference}{Volume 2 (2018), Article Nr. 2}

\acknowledgement{The first  author acknowledges support from the CysMoLog ``d\'efi
scientifique \'emergent'' of the Universit\'e de Rennes 1.}

 \usepackage[all]{xy}

\allowdisplaybreaks

\newtheorem{theorem}{Theorem}[section]

\newtheorem{proposition}[theorem]{Proposition}
\newtheorem{lemma}[theorem]{Lemma}

\newtheorem{remark}[theorem]{Remark}

\numberwithin{equation}{section}

\DeclareMathOperator{\Disc}{Disc}

\DeclareMathOperator{\Jac}{Jac}

\DeclareMathOperator{\Pic}{Pic}

\DeclareMathOperator{\Prym}{Prym}

\DeclareMathOperator{\Spec}{Spec}

\def\P{\mathbf{P}}

\def\Z{\mathbf{Z}}

\def\ie{\textit{i.e.} }

\renewcommand{\geq}{\geqslant}

\renewcommand{\leq}{\leqslant}

\newcommand{\AAb}{\mathbb{A}}

\newcommand{\cc}{\mathcal{C}}
\newcommand{\dD}{\mathcal{D}}

\newcommand{\pp}{\mathcal{P}}

\newcommand{\sS}{\mathcal{S}}

\newcommand{\xx}{\mathcal{X}}

\newcommand{\tD}{\tilde{D}}
\newcommand{\tC}{\tilde{C}}

\newcommand{\tpi}{\tilde{\pi}}
\newcommand{\dpi}{\pi^{\textrm{\tiny def}}}
\newcommand{\tio}{\tilde{\iota}}

\begin{document}


\maketitle



\begin{prelims}


\def\abstractname{Abstract}
\abstract{Given a generic degree-$2$ cover of a genus 1 curve $D$ by a non-hyperelliptic
genus~$3$ curve $C$ over a field~$k$ of characteristic different from~$2$, we produce an explicit genus~$2$ curve~$X$ such that $\Jac(C)$ is isogenous to
$\Jac(D) \times \Jac(X)$. This construction can be seen as a degenerate case of
Bruin's result \cite{bruin}.}



\vspace{0.05cm}

\languagesection{Fran\c{c}ais}{%

\textbf{Titre. Sur la vari\'et\'e de Prym des rev\^etements de genre $3$ des courbes de genre $1$} \commentskip \textbf{R\'esum\'e.} \'Etant donn\'e un rev\^etement de degr\'e $2$ g\'en\'erique d'une courbe $D$ de genre $1$ par une courbe non-hyperelliptique $C$ de genre $3$ sur un corps $k$ de caract\'eristique diff\'erente de $2$, nous produisons une courbe explicite $X$ de genre $2$ dont la jacobienne $\Jac(C)$ est isog\`ene \`a $\Jac(D)\times\Jac(X)$. Cette construction peut \^etre vue comme un cas d\'eg\'en\'er\'e du r\'esultat de Bruin \cite{bruin}.}

\end{prelims}


\newpage

\setcounter{tocdepth}{1}
\tableofcontents

\section{Introduction}
Let $C$ be a (smooth, projective, absolutely irreducible) curve of genus $g_C \geq 2$ over  a field $k$. Let us assume that $C$ admits a degree-$n$ map $\pi$ to a curve $D/k$ of genus $g_D$ with $0<g_D < g_C$. We then know that $\Jac(C)$ is isogenous to $\Jac(D) \times A$ and one would like to get as much information as possible on $A$ from $\pi$. For instance, if $C$ is given by explicit equations, is it possible to find a curve $X/k$ given by explicit equations such that $A$ is isogenous over $k$ to $\Jac(X)$?

There is an extensive literature dedicated to the case $g_C=2$ for which $D$ and $X$ are genus~$1$ curves, both for its applications to cryptography (see \cite{FS11}, \cite{cosset} and \cite{langen2}) and  for its interests in pure mathematics \cite{kaniarith}. The general theory is well understood (see for instance \cite{kani}) and explicit formulas have been worked out when $n=2$ (going back to the work of Jacobi on abelian integrals, see the references in \cite[p.395]{baker} or \cite{leprevost}), $n=3$ (see \cite{goursat}, \cite{kuhn} or the appendix of \cite{BHLS15}), $n=4$ (see \cite{bolzan4} and \cite{bruinn4}), $n=5$ (see \cite{MSV09}) and more generally when $n \leq 11$ (see \cite{kumar}).

From now on, let $n=2$. When $C$ is given by a hyperelliptic equation, we can describe the map $\pi$ as the quotient by the involution  $(x,y) \mapsto (-x,y)$. Then $C : y^2=f(x^2)$ where $f$ is a polynomial of degree $g_C+1$ and hence $D : y^2=f(x)$ is a curve of genus $g_D=\lceil \frac{g_C+1}{2} \rceil -1$. One can moreover take advantage of the existence of the hyperelliptic involution to construct the curve $X : y^2=x f(x)$ which is the quotient of $C$ by $(x,y) \mapsto (-x,-y)$. It is easy to see by the action on the differentials that $\Jac(C) \sim \Jac(D) \times \Jac(X)$.

Apart from the hyperelliptic cases described above, for a generic curve $C$, a curve $X$ has been made explicit only in few cases. One is when $D$ is hyperelliptic and $C$ is a degree-2 cover either \'etale or ramified at two points, see respectively \cite[p.346]{mumfordprym} and \cite{levin}. In \cite{bruin}, one can find  the only case where $D$ is neither elliptic nor hyperelliptic, namely $g_C=5$ and $g_D=3$.  This work relies on beautiful geometric constructions sketched in \cite[VI exercise F]{harris}. We will use this construction to get our result.

The present article deals with the case where $C$ is a non-hyperelliptic curve of genus $3$ with a degree-$2$ map to a genus $1$ curve $D$ over a field $k$ of characteristic different from $2$. It finds its motivation in the article \cite{shieh} where cohomological techniques are derived to compute the Weil polynomial of $C$. Although the techniques developed there remain interesting on their own, one byproduct of the present article is that the computation of the Weil polynomial of $C$ can be reduced to computations in genus 1 and 2.

Note that any degree-2 cover of a genus $1$ curve by a non-hyperelliptic genus $3$  curve $C$ can be written in the form $y^4- h(x,z) y^2 + r(x,z)=0$ with $r$ a degree $4$ polynomial, since we can assume that the involution  of $C$ which defines the cover is given by $(x:y:z) \mapsto (x:-y:z)$.  Our main result gives an equation of $X$ when $r$ splits as a product of degree 2 factors.

\begin{theorem} \label{th:main}
 Let $C$
 be a smooth, non-hyperelliptic genus $3$ curve defined by
  $$C : y^4 - h(x,z) \, y^2 + f(x,z) \,  g(x,z)=0$$ 
in $\P^2$ where $$f=f_2 x^2+ f_1 x z + f_0 z^2, \quad g=g_2 x^2+g_1 x z + g_0 z^2, \quad h = h_2 x^2 + h_1 xz +h_0 z^2$$ are homogeneous degree-2 polynomials over a field $k$ of characteristic different from $2$.  The involution $(x:y:z) \mapsto (x:-y:z)$ induces a degree-2 cover $\pi$ to the genus $1$ curve  $$D : y^2 - h(x,z) \, y + f(x,z) \, g(x,z)=0$$
 in the weighted projective space $\P{(1,2,1)}$.
Let $$A = \begin{bmatrix} f_2 & f_1 & f_0 \\ h_2 & h_1 & h_0 \\ g_2 & g_1 & g_0 \end{bmatrix}$$
and assume that $A$ is invertible. Let $${A^{-1}} = \begin{bmatrix} a_1 & b_1 & c_1 \\ a_2 & b_2 & c_2 \\ a_3 & b_3 & c_3 \end{bmatrix}.$$ Then  $\Jac(C) \sim \Jac(D) \times \Jac(X)$ with $X \simeq y^2 = b \cdot (b^2-ac)$ in $\P{(1,3,1)}$ where 
$$a=a_1+2 a_2 x+ a_3 x^2, \quad b=b_1 + 2 b_2 x + b_3 x^2, \quad c=c_1+2 c_2 x+ c_3 x^2.$$
\end{theorem}

\begin{remark}{\rm
The geometric hypothesis (\textit{i.e.} which is still true over $\bar{k}$) that $\det(A) \ne 0$ is satisfied if $C$ has automorphism group isomorphic to $\Z/2\Z$. Indeed since $C$ is non-singular over $\bar{k}$, we can assume that $f=xz$, $g=x^2 + g_1 x z+ z^2$ ($fg$ has no double root). If $A$ is non-invertible, we have for instance that $h= \lambda f + \mu g$ but then the morphism $(x:y:z) \mapsto (z:y:x)$ is also an involution. In that case, the Jacobian is going to decompose further in a product of three genus $1$ curves which are easily determined.}
\end{remark}
\begin{remark}{\rm
The arithmetic hypothesis that $r=f g$ over $k$ is however necessary in our proof and we do not know how to remove it. Note that it can however always be obtained over a finite extension of $k$.}
\end{remark}

After we wrote this article, it was brought to our attention by Victor Enolski and Yuri Fedorov that such a case was already handled by K\"otter in 1892
(see \cite{kotter} and also \cite[Eq.(2--11)]{fedorov}) in an analytic setting. The proof we propose here is completely different as it is algebraic and works over any field (of characteristic not $2$). The shape of the equation we get  for $X$ is also much simpler 
and in particular does not involve any field extension of $k$.

The general strategy is somehow new in this context although it will probably appear very classical to algebraic geometers. The idea is to move from the smooth cover $\pi : C \to D$ to a singular one $\tilde{\pi} : \tilde{C} \to \tilde{D}$ of arithmetic genus $5$ and $3$, with $\tilde{C}$ birational to $C$. The reason to do so is the following. As $n=2$, one can look at $A$ as the Prym variety $P(C/D)$, \textit{i.e.} the connected component of the identity of $\ker(\pi_* : \Jac(C) \to \Jac(D))$. Now, the Prym variety of a degree-2 cover $\pi$ naturally inherits a principal polarization from $\Jac(C)$ when $\pi$ is \'etale (see for instance \cite[Cor.2]{mumfordprym}). Our original smooth cover $\pi$ \emph{is not} \'etale but the cover $\tilde{\pi}$ will only be ramified over the singular points of $\tilde{D}$. In this case (more precisely in the case of \emph{allowable covers}, see below), Beauville showed that the generalized Prym variety of $\tilde{\pi}$ is an abelian surface $P(\tC/\tD)$, isogenous to $P(C/D)$, and naturally principally polarized. The abelian variety $P(\tC/\tD)$ is therefore the Jacobian of a (possibly reducible) genus $2$ curve $X$.  It then remains to find $X$ explicitly. In order to do so, we look at the cover $\tilde{C} \to \tilde{D}$ as the degeneration of a family of unramified degree-$2$ covers $\mathcal{C} \to \mathcal{D}$ of smooth genus $3$ curves by smooth degree $5$ curves. As we said, for the smooth generic fiber, Bruin \cite{bruin} gives  explicitly the hyperelliptic curve $\mathcal{X}$ whose Jacobian is the Prym of the cover. Using a particular deformation, we show that we can specialize the equation of $\mathcal{X}$ to find $X$.

\subsection*{Acknowledgements}
We want to thank Jeroen Sijsling for his comments on an earlier version of the paper and Abhinav Kumar, Victor Enolski and Yuri Fedorov for their references on the subject.

\section{From a singular cover to the original cover}
Let the notation and hypotheses be as in Theorem \ref{th:main} and let us consider the curve in $\P^4$ defined by
$$\tilde{C} : \begin{cases}
x_1 x_2 & = f(u,v) \\
x_2^2 + x_1 x_3 & = h(u,v) \\
x_2 x_3 &= g(u,v)
\end{cases}$$
 Multiplying  the second equation by $x_2^2$ and substituting the first and last equation we get  our initial curve
$$C : x_2^4 - h(u,v) x_2^2 + f(u,v) g(u,v)=0.$$
Hence we get an isomorphism between $C$ and $\tC$ on the locus $x_2 \ne 0$. \\

Now let us rewrite $\tilde{C}$ using the matrix $A$. We have
$$\tilde{C} :  \textrm{Id} \cdot \begin{bmatrix} x_1 x_2 \\ x_2^2+x_1 x_3 \\ x_2 x_3 \end{bmatrix} = { A} \cdot \begin{bmatrix}  u^2 \\ u v \\ v^2 \end{bmatrix}.$$
As $A$ is invertible, we get:
$$\tilde{C} : { A^{-1}}  \begin{bmatrix} x_1 x_2 \\ x_2^2+x_1 x_3 \\ x_2 x_3 \end{bmatrix} = \begin{bmatrix}  u^2 \\ u v \\ v^2 \end{bmatrix}.$$
Let  $(q_1(x_1,x_2,x_3),q_2(x_1,x_2,x_3),q_3(x_1,x_2,x_3))$ be the left-hand side vector, we get that
$$\tilde{C} : \begin{cases}
q_1 & = u^2 \\
q_2 & = uv \\
q_3 &= v^2
\end{cases}$$ is a degree $2$ cover $\tpi$ of the curve $\tD : q_2^2=q_1 q_3$. In order to see that the covers $\pi$ and $\tpi$ are birationally equivalent, let us consider the automorphism $(x_2:u:v) \mapsto (-x_2:u:v)$ of $C$ which induces the cover $\pi: C \to D$. On $\tilde{C}$, this automorphism induces the automorphism $(x_1 : x_2 : x_3 : u : v) \mapsto (-x_1 : -x_2 : -x_3 :u :v)$ since $x_1 x_2=f(u,v)$ and $x_2 x_3=g(u,v)$. But $(-x_1 : -x_2 : -x_3 :u :v)= (x_1:x_2:x_3:-u:-v)$ and hence the automorphism is $\tio$ which induces the cover $\tpi : \tilde{C} \to \tD$. In particular $D$ is birationally equivalent to $\tD$ over~$k$.

To conclude, let us study the singular points of $\tilde{C}$. Since $C$ is non-singular, they can appear only when $x_2=0$. But  from the first form of the equations of $\tC$, we see that $f(u,v)=g(u,v)=0$. Since $C$ is non-singular, $f$ and $g$ have no common root, hence $u=v=0$. The possible singular points are then only  the points $p_0=(0:0:1:0:0)$  and $p_{\infty}=(1:0:0:0:0)$. The tangent cone at the point $p_0$ has equations $x_1=x_2=f(u,v)=0$ while the tangent cone at the point $p_{\infty}$ has equations $x_2=x_3=g(u,v)=0$ (see for instance \cite[III, \S~3]{redbook} or \cite[p.485]{cox2}). Since $f$ and $g$ have no double root, it follows that these points are nodal singularities. Therefore we see that~:

\medskip

\noindent (\textasteriskcentered)
The fixed points of the involution $\tio$ are exactly the singular points,
which are nodal, and $\tio$ preserves the branches at these points.

\smallskip

This condition is phrased in Beauville \cite[p.157]{beauville}.
Since the map $\tpi$ is ramified at a point $p=(a:b:c:d:e)$ if and only if $\tio(p)=p$, \ie if and only if $d=e=0$ and $ab=0, bc=0, b^2+ac=0$, we see moreover that
$\tpi$ is ramified only at the points $p_0$ and $p_{\infty}$, in other words:

\medskip

\noindent (\textasteriskcentered\textasteriskcentered)
The cover $\tpi$ is unramified away from the singular locus.

\medskip

In Donagi and Livn\'e \cite{donagi}, covers satisfying conditions (\textasteriskcentered) and (\textasteriskcentered\textasteriskcentered)
are called {\em allowable} (also \emph{admissible} by other authors) and we keep this terminology. Finally we
sum up the preceding discussion.

\begin{proposition} \label{prop:singu}
The cover $\tpi : \tC \to \tD$ is birationally equivalent to the cover
$\pi : C\to D$ and is allowable.
\end{proposition}

Let us associate to the cover $\tilde{C} \to \tD$  the connected component containing $0$  of the kernel of the norm map $\tpi_* : \Jac(\tilde{C}) \to \Jac(\tD)$ between generalized Jacobians. We denote it $P(\tC/\tD)$ and call it the \emph{generalized Prym variety}.
Since the cover $\tilde{C} \to \tD$ is allowable, the results \cite[Prop.3.5]{beauville} and \cite[Lem.1]{donagi} show that $P(\tC/\tD)$ is an abelian variety which is isogenous to the classical Prym variety $P(C/D)$ of the ramified cover $C \to D$. Moreover the kernel of the isogeny is contained in the group generated by degree $0$  divisors supported on the singular points of $\tilde{C}$ and is therefore defined over $k$. We hence get:
\begin{lemma} \label{lem:link}
The Jacobian of $C$ is isogenous over $k$ to $D \times P(\tC/\tD)$.
\end{lemma}
As mentioned in the introduction, what we have gained by moving from $P(C/D)$ to $P(\tC/\tD)$ is that the latter is principally polarized \cite[Th.3.7]{beauville}. From Weil (see the version in ~\cite[Thm.~3.1]{GGR}), one knows that every principally polarized abelian surface is either the Jacobian of a curve, the product of two elliptic curves (with the product polarization) or the restriction of scalars of an elliptic curve over a quadratic extension of $k$. Under our hypotheses, we are now going to recover $P(\tC/\tD)$ as the Jacobian of a genus $2$ curve $X$.
 
\section{The singular cover as a special fiber} \label{sec:gt}
  

In order to find explicitly the curve $X$, we will `embed' our singular cover $\tilde{C} \to \tD$ into a flat family whose generic fibre is non-singular and use the beautiful result of \cite{bruin} which we recall here.
\begin{proposition} \label{prop:bruin}
Let $Y \to Z$ be an unramified double cover of a genus $5$ curve over a
non-hyperelliptic curve of genus $3$ over a field $k$ of characteristic different from $2$. Then such a cover can be written $Z : Q_1 Q_3=Q_2^2$ where $Q_i$ ($i=1,2,3$) are quadratic forms, and 
$$Y : \begin{cases}
Q_1 & = u^2 \\
Q_2 & = uv \\
Q_3 &= v^2
\end{cases}.$$
Moreover the Prym variety $\Prym(Y/Z)$ (as a principally polarized abelian surface over $k$) is  the Jacobian of the hyperelliptic curve $H$ which is the projective closure of the curve with equation $y^2 = -\det(Q_1+2 x Q_2 + x^2 Q_3)$,
where we identify  $Q_i$ with the symmetric $3 \times 3$ matrices $((\partial^2 Q_i / \partial x_j \partial x_k))_{1 \leq j,k \leq 3}$.  
\end{proposition}

\begin{remark}{\rm
Bruin's result is stated in characteristic $0$ only. But all his arguments work as well in positive characteristic different from $2$.}
\end{remark}

The \emph{discriminant} $\Disc(F)$ of a degree $4$ homogeneous polynomial  $F(x_1,x_2,x_3)$ over a field $K$ of characteristic different from $2$ is defined as the multivariate resultant of the partial derivatives of $F$ (see \cite[p.426]{gelfand}). It is a polynomial of degree $27$ in the coefficients of a generic quartic and $\Disc(F)=0$ if and only if the curve defined by $F$ is singular \cite[Chap.13.1.D]{gelfand}.

If we restrict the discriminant to quartics $F$ of the form $Q_2^2-Q_1 Q_3$ where $$Q_i=\sum_{i_1+i_2+i_3=2} a_{i_1i_2i_3}^{(i)} x_1^{i_1} x_2^{i_2} x_3^{i_3}$$ are degree $2$ homogeneous polynomials over  $K$, we get a homogeneous polynomial~$H$ of  degree $2 \cdot 27=54$ in the coefficients of the $Q_i$. Note that $H$ is nonzero since for instance  the quartic $$x_1^4-x_2^4+x_3^4=(x_1^2)^2-(x_2^2+x_3^2) \cdot (x_2^2-x_3^2)$$ has a discriminant equal to $-2^{40}$. The equation $H$ defines a hypersurface in $$V=\Spec K[a_{i_1i_2i_3}^{(i)}] \simeq \AAb^{18}.$$ 
The point $x_{\tD}$ representing the singular quartic $\tD = q_2^2-q_1 q_3$ belongs to $H$ 
and the  pencil  $$(Q_1(\lambda),Q_2(\lambda),Q_3(\lambda))=(q_1+\lambda q_1',q_2+\lambda q_2',q_3+\lambda q_3'), \quad \lambda \in \AAb^1$$ 
such that $$q_1' = (x_2^2+x_3^2)-q_1,\quad q_2'=x_1^2-q_2, \quad q_3'=(x_2^2-x_3^2)-q_3$$
is not contained in $H$ so its  generic element is smooth. 

Now, let $\sS$ be the spectrum of the discrete valuation ring $R=k[[\epsilon]]$ with special fiber $k$ and generic fiber $K$. Over $\sS$ we can define a cover $\dpi :  \cc \to \dD$ by
$$\cc :  \begin{cases}
Q_1(\epsilon) & = u^2 \\
Q_2(\epsilon) & = uv \\
Q_3(\epsilon) &= v^2
\end{cases} \to \dD : Q_2(\epsilon)^2-Q_1(\epsilon) Q_3(\epsilon)=0.$$
Clearly $\dpi$ is a deformation of $\tpi : \tC \to \tD$ and by construction its generic fiber $\pi_K : \cc_K \to \dD_K$ is smooth over $K$. The Prym varieties
of the fibres of $\dpi$ fit in a flat family which we call the \emph{Prym scheme}
$\pp= P(\cc \to \dD)$. The details of its construction are to be found in
\cite{beauville}, Section~6, especially~(6.2); we recall them briefly.
The morphism
$\dpi_* : \Pic^0(\cc) \to \Pic^0(\dD)$ is surjective with smooth target
hence it is flat, and its fibres are smooth. Therefore the group scheme
$\ker(\dpi_*)\to\sS$ is smooth and we let $\pp\to \sS$ denote the connected
component of the unit section. This is the Prym scheme; it has proper fibres
hence it is an abelian scheme. Using Proposition \ref{prop:bruin} as well
as the fact that the functor $\Pic^0$
commutes with base change, we get the following lemma.

\begin{lemma} \label{lem:pic}
The special fiber of $\pp$ is $P(\tC/\tD)$ and the generic fiber is
$P(\cc_K/\dD_K)= \Jac(X_K)$ where
$X_K : y^2=-\det(Q_1(\epsilon)+2 x Q_2(\epsilon) + x^2 Q_3(\epsilon))$.
\end{lemma}

Let $\xx/\sS$ be the projective closure of  $\Spec(R[x,y]/(y^2+\det(Q_1(\epsilon)+2 x Q_2(\epsilon) + x^2 Q_3(\epsilon))$ in the weighted projective space $\P^{(1,3,1)}$, in other words the projective hypersurface described by the homogenized
equation. The special fiber $X= \xx \otimes k$ is the completion of 
$$y^2= - \det(Q_1(0) + 2 x Q_2(0) +x^2 Q_3(0))= -\det(q_1 +2 x q_2 + x^2 q_3).$$
We are going to prove the following lemma.
\begin{lemma}
The scheme $\xx$ is smooth over $\sS$ with geometrically connected fibres.
\end{lemma}
\begin{proof}
As we are considering our schemes in the weighted projective space $\P{(1,3,1)}$, the point at infinity is smooth and we only have to consider the affine part in what follows. Since the generic fiber $X_K$ is smooth, we need to consider the singular points of the special fiber $X$. They correspond to points where $F(x)=\det( q_1 +2 x  q_2 + x^2 q_3)$ has a (projective) root with multiplicity greater than~1.  With the notation of Theorem \ref{th:main}, since 
$$\begin{bmatrix} q_1 \\ q_2 \\ q_3 \end{bmatrix} = \begin{bmatrix} a_1 & b_1 & c_1 \\ a_2 & b_2 & c_2 \\ a_3 & b_3 & c_3 \end{bmatrix} \begin{bmatrix} x_1 x_2 \\ x_2^2+x_1 x_3 \\ x_2 x_3 \end{bmatrix}$$ 
we have $F=b \cdot (b^2-ac)$. We are actually going to prove that $F$ has no multiple root. As this is a geometric problem, we can assume that we work over an algebraically closed field. In that case, since $C$ is assumed to be non-singular, we can take $f(x,z)=xz$ and $g(x,z)=g_2 x^2 + g_1 x z+ z^2$.  A computation with a computer algebra system such as \texttt{Magma} shows that 
$$\textrm{Disc}(F) = \frac{4 \cdot   g_2 \cdot (g_2-g_1^2/4)^2 \cdot \textrm{Disc}(h^2-4 f g)}{\det(A)^{18}}.$$
Since the curve $D$ is assumed to be non-singular, we see that $4 f g - h^2$ has no multiple root, hence $\textrm{Disc}(h^2-4 f g) \ne 0$. The curve $C$ being non-singular, one has that $g_2 \ne 0$ (otherwise $(1:0:0)$ would be a singular point) and one finds that $g$ does not have a double root, hence $g_2-g_1^2/4 \ne 0$. We hence get $\textrm{Disc}(F) \ne 0$.

Finally, since $\xx\to\sS$ is smooth and proper, the number of connected components of the fibres is constant on $\sS$, hence equal to $1$.
\hfill $\Box$
\end{proof}

With this result on $\xx$,  we have the right hypotheses to apply the following proposition (see for instance \cite[Th. 9.5/1]{BLR}).

\begin{proposition} \label{prop:artin}
$\Pic^0(\xx)$ is a N\'eron model of $\Jac(X_K)$.
\end{proposition}

We can now conclude. By Lemma \ref{lem:pic}, the N\'eron model of the Jacobian of $X_K$ is the N\'eron model of the generic fiber $P(\cc_K/\dD_K)$ of $\pp$. Since the special fiber of $\pp$ is $P(\tC/\tD)$ which is an abelian variety (the cover~$\tpi$ being allowable), we see that the N\'eron model of $P(\cc_K/\dD_K)$ is $\pp$. Hence the N\'eron model of $\Jac(X_K)$ is $\pp$ and by Proposition \ref{prop:artin} we have $\pp = \Pic^0(\xx) $. Taking the special fiber gives:
$$\pp \otimes k = P(\tC/\tD)=  \Pic^0(\xx) \otimes k = \Jac(\xx \otimes k) = \Jac(X).$$
We have therefore obtained a description of $P(\tC/\tD)$ which is isogenous to $P(C/D)$ and this concludes the proof of Theorem \ref{th:main}.

\providecommand{\bysame}{\leavevmode\hbox to3em{\hrulefill}\thinspace}
%
%

\bibliographystyle{amsalpha}
\bibliographymark{References}

\end{document}